\documentclass[12pt,twoside]{article}
\usepackage{amssymb,amsmath,amsthm,latexsym}
\usepackage{amsfonts}
\usepackage{enumerate}
\usepackage[demo]{graphicx}
\usepackage{caption}
\usepackage{subcaption}
\usepackage{relsize}

\usepackage{tikz}
\usepackage{mathrsfs}
\usepackage{doi}
\usetikzlibrary{arrows.meta,positioning, quotes}
\usepackage{calc}

\usetikzlibrary{decorations.markings}
\tikzstyle{vertex}=[circle, draw, inner sep=0pt, minimum size=6pt]
\newcommand{\vertex}{\node[vertex]}
\newtheorem{thm}{Theorem}[section]

\newtheorem{lem} [thm]{Lemma}
\newtheorem{prop} [thm]{Proposition}
\theoremstyle{definition} % Definitions, examples, remarks, algorithms should be in roman type, not italic.

\newtheorem{ex}[thm]{Example}
\newtheorem{rmk}[thm] {Remark}
\newtheorem{defn}[thm]{Definition}

\numberwithin{equation}{section}

\newcommand\sgn{\operatorname{sgn}}
\newcommand\bdim{\operatorname{bdim}}
\newcommand\sbdim{\operatorname{sbdim}}

\begin{document}
	\label{'ubf'}
	\setcounter{page}{1} %Put here the starting page number

	\markboth {\hspace*{-9mm} \centerline{\footnotesize \sc
			% Put here the left page top label
			Vector Valued Switching in Signed Graphs}
	}
	{ \centerline {\footnotesize \sc Hameed, Albin,  Germina, Zaslavsky
			}
	}
	\begin{center}
		{
			\Large \textbf{Vector Valued Switching in Signed Graphs
				}
			}

			\bigskip
			 Shahul Hameed K \footnote{\small Department of Mathematics, K M M Government\ Women's\ College, Kannur - 670004,\ Kerala,  \ India.  \textbf{Email:}\texttt{ shabrennen@gmail.com}}\, Albin Mathew \footnote{\small Department of Mathematics, Central University of Kerala, Kasaragod - 671316,\ Kerala,\ India. \textbf{Email:}\texttt{ albinmathewamp@gmail.com, srgerminaka@gmail.com}}\, Germina K A $^{2}$\, Thomas Zaslavsky \footnote{\small Department of Mathematics and Statistics, Binghamton University (SUNY), Binghamton, NY 13902-6000, USA. \textbf{Email:} \texttt{zaslav@math.binghamton.edu}}
			\\

\end{center}

\thispagestyle{empty}
\begin{abstract}
	A signed graph is a graph with edges marked positive and negative; it is unbalanced if some cycle has negative sign product.  We introduce the concept of vector valued switching function in signed graphs, which extends the concept of switching to higher dimensions. Using this concept, we define balancing dimension and strong balancing dimension for a signed graph, which can be used for a new classification of degree of imbalance of unbalanced signed graphs. We provide bounds for the balancing and strong balancing dimensions, and calculate these dimensions for some classes of signed graphs.
\end{abstract}



%\tableofcontents

%%%%%%%%%%%%%%%%
\section{Motivational Background and Introduction}

The concept of switching in signed graphs was introduced by Zaslavsky in \cite{tz1}. Given a signed graph $\Sigma=(G,\sigma)$ where $G=(V,E)$ is the underlying graph (which we assume is simple) and $\sigma:E\rightarrow \{-1,1\}$ is the signing function, by switching   $\Sigma$  to a signed graph $\Sigma^{\zeta}=(G,\sigma^\zeta)$ using a switching function $\zeta:V\rightarrow \{-1,1\}$, we mean the edge signing of $\Sigma^{\zeta}$ satisfies the condition $\sigma^\zeta(uv)=\sigma(uv)\zeta(u)\zeta(v)$.  Switching does not change the signs of cycles.  We say two signed graphs $\Sigma_1$ and $\Sigma_2$ are switching equivalent  if one of them can be switched from the other.

Given a cycle $C$ in a signed graph, the sign of this cycle $\sigma(C)$ is defined as the product of the edge signs on it. If $\sigma(C)=1$, we say that the cycle $C$ is positive. A signed graph is said to be balanced if all cycles in it are positive. There are various characterizations of balanced signed graphs; one of them is by switching (e.g., see \cite{tz}), as follows.
\begin{thm} \label{thm1} A signed graph $\Sigma=(G,\sigma)$ is balanced if and only if it can be switched to an all positive signed graph.
\end{thm}

An undirected graph $G$ can be considered as an all positive signed graph.  This is a more restrictive property than balance because no switching is required to make all edges positive, but balance is still quite restrictive because it requires that all cycle signs be positive.  Indeed, balanced signed graphs are the signed graphs that are the most like unsigned graphs.

If $\Sigma = (G,\sigma)$ is a signed graph, then $-\Sigma = (G,-\sigma)$ is the same signed graph with all signs reversed.  For example, $-G$ means $G$ with all negative edges.  We say $\Sigma$ is antibalanced when $-\Sigma$ is balanced.  It is easy to see that $-(\Sigma^\zeta) = (-\Sigma)^\zeta$, so by Theorem \ref{thm1} $\Sigma$ is antibalanced if and only if it switches to all negative signs.

Motivated from the above theorem, as the product $\zeta(u)\zeta(v)$ can be viewed as the inner product of $\zeta(u)$ and $\zeta(v)$ on $\mathbb{R}$,  we frame the following definitions to classify unbalanced signed graphs extending the concept of switching to a higher dimension. In what follows, $\Omega=\{-1,0,1\} $ and the inner product used is the same as that on $\mathbb{R}^k$ restricted to $\Omega^k$.
\begin{defn}[Vector Valued Switching or $k$-Switching]  Let $\Sigma=(G,\sigma)$ be a given signed graph  where $G=(V,E)$.  A vector valued switching function is a function $\zeta:V\rightarrow \Omega^k \subset \mathbb{R}^k$ such that $\langle \zeta(u), \zeta (v)\rangle\neq 0$ for all edges $uv\in E$.  The switched signed graph $\Sigma^\zeta=(G,\sigma^\zeta)$ has the signing $$\sigma^\zeta(uv)=\sigma(uv) \sgn (\langle \zeta(u), \zeta (v)\rangle).$$
The switching considered so far in the literature, from now onwards will be referred to as $1$-switching and the generalized switching introduced here will be mentioned as $k$-switching.
\end{defn}

\begin{rmk}\label{zero}
The zero vector is a possible value of $\zeta(v)$, but only if $v$ is an isolated vertex.  Although an isolated vertex in $1$-switching can take the value $0$ so that the usual switching is not precisely the same as $1$-switching, the difference is not important.
\end{rmk}

The product of a $1$-switching function $\eta$ and a $k$-switching function $\zeta$ is defined by $(\eta\zeta)(v) = \eta(v)\zeta(v)$.

\begin{thm}\label{thm2}
	Let $k,k'\geq1$.  A $k$-switching function $\zeta$, a $k'$-switching function $\zeta'$, and a $1$-switching function $\eta$ satisfy $(\Sigma^\zeta)^{\zeta'}=(\Sigma^{\zeta'})^\zeta$, $\Sigma^{\eta\zeta}=(\Sigma^\eta)^\zeta=(\Sigma^\zeta)^\eta$, $\Sigma^{-\zeta}=\Sigma^\zeta$, $(\Sigma^\zeta)^\zeta=\Sigma$, and $-(\Sigma^\zeta) = (-\Sigma)^\zeta$.
\end{thm}

%%%%%%%%%%%%%%%%
\section{Balancing and Strong Balancing Dimensions}

\subsection{Definition and elementary properties}

\begin{defn}[Balancing Dimension]  Let $\Sigma=(G,\sigma)$ be a given signed graph  where $G=(V,E)$. We say that the balancing dimension of $\Sigma$ is $k$ and write it as $\bdim(\Sigma)=k$, if $k\ge 1$ is the least integer such that a vector valued switching function $\zeta:V\rightarrow \Omega^k\subset \mathbb{R}^k$ switches $\Sigma$ to an all positive signed graph.  
We call such a $k$-switching function $\zeta$ a positive $k$-switching function (briefly a $k$-positive function) for $\Sigma$.
\end{defn}

\begin{lem}\label{up}
A signed graph $\Sigma$ has a $k$-positive function for every $k\geq\bdim(\Sigma)$.
\end{lem}
\begin{proof}
Let $j<k$ and let $\zeta: V \rightarrow \Omega^j$ be a $j$-positive function for $\Sigma$.  Define $\zeta'(v) = (\zeta_1(v),\ldots,\zeta_j(v),0,\ldots,0) \in \Omega^k$.  Then $\zeta'$ is a $k$-positive function for $\Sigma$.  In particular, take $j=\bdim(\Sigma)$.
\end{proof}

\begin{defn}[Strong Balancing Dimension]  Let $\Sigma=(G,\sigma)$ be a given signed graph  where $G=(V,E)$. We say that the strong balancing dimension of $\Sigma$ is $k$ and write it as $\sbdim(\Sigma)$, if $k\ge 1$ is the least integer such that there is an injective vector valued switching function $\zeta:V\rightarrow \Omega^k$ which switches $\Sigma$ to an all positive signed graph.
However, in case $\Sigma$ is all positive, we define $\sbdim(\Sigma)=1$.

We call such a $k$-switching function $\zeta$ an injective positive $k$-switching function (briefly, a strongly $k$-positive function) for $\Sigma$.
\end{defn}

We chose to study injectivity because by allowing higher dimensional switching we open the door to new variations on the definition of a switching function, and injectivity seemed an interesting and attractive such variation.

\begin{thm}\label{thm4}
$\bdim(\Sigma)=1$ if and only if $\Sigma$ is balanced.
Contrastingly, $\sbdim(\Sigma)=1$ if and only if $\Sigma = K_1$, $K_1 \cup K_1$, $K_1 \cup K_1 \cup K_1$, $-K_2$, or $-K_2 \cup K_1$.
\end{thm}
\begin{proof} 
If $\sbdim(\Sigma)=1$, then $\Sigma$ has at most $3$ vertices and if it has $3$, the one with $\zeta(v)=0$ must be isolated.  If there are two non-isolated vertices, they must be negatively adjacent.
\end{proof}
We note that $\sbdim(\Sigma)=1$ has a few more examples than would exist if $1$-switching were identical to ordinary switching.

\begin{ex}\label{ex1}
$\bdim(C_4^-)=\sbdim(C_4^-)=2$.
Take $\zeta(v_1)=(-1,0)$, $\zeta(v_2)=(1,-1)$, $\zeta(v_3)=(0,-1)$ and $\zeta(v_4)=(-1,-1)$ to see that both dimensions are $2$.
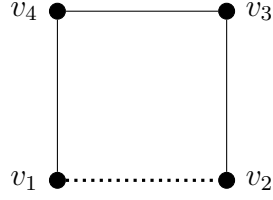
\begin{figure}[h]
	\centering
	\begin{tikzpicture}[x=0.75cm, y=0.75cm]
		\vertex[fill] (v) at (0,0) [label=left:$v_1$] {};
		\vertex[fill] (w) at (3,0) [label=right:$v_2$] {};
		\vertex[fill] (x) at (0,3) [label=left:$v_4$] {};
		\vertex[fill] (y) at (3,3) [label=right:$v_3$] {};
		\path[very thick, dotted]
		(v) edge (w)

		;
		\path
		(x) edge (v)
		(w) edge (y)
		(y) edge (x)
		;
	\end{tikzpicture}
	\caption{The negative cycle $C_4^-$}
	\label{fig1}
\end{figure}
\end{ex}

\begin{thm}\label{thm5} $\bdim$ is $1$-switching invariant.
\end{thm}
\begin{proof} This is a consequence of Theorem \ref{thm2}.
\end{proof}

\begin{rmk}
	Strong balancing dimension need not be $1$-switching invariant. Consider the signed graph shown in Figure~\ref{fig2}.
\begin{figure}[h!]
	\centering
	\[\begin{tikzpicture}[x=0.6cm, y=0.6cm]
		\vertex[fill] (a) at (-1,6) [label=left:$v_1$] {};
		\vertex[fill] (b) at (3,6) [label=right:$v_2$] {};
		\vertex[fill] (c) at (1,4) [label=left:$v_3$] {};
		\vertex[fill] (d) at (-1,2) [label=left:$v_4$] {};
		\vertex[fill] (e) at (3,2) [label=right:$v_5$] {};
		\vertex[fill] (f) at (0,0) [label=left:$v_6$] {};
		\vertex[fill] (g) at (2,0) [label=right:$v_7$] {};
		\path[very thick, dotted]
		(a) edge (c)
		(b) edge (c)
		(d) edge (c)
		(e) edge (c)
		(d) edge (f)
		(e) edge (g)
		(f) edge (g)

		;
	\end{tikzpicture}\]
	\caption{A signed graph $\Sigma$ with $\sbdim(\Sigma)=3$}
	\label{fig2}
\end{figure}
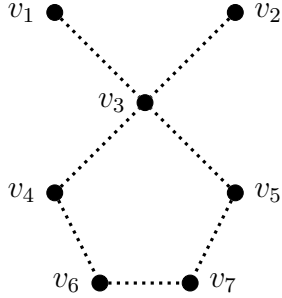

 For any non-zero $\alpha \in\Omega^2$, the cardinality of the set $\{\beta\in\Omega^2:\langle\alpha,\beta\rangle<0\}$ is 3. Thus, there does not exist an injective switching function from $V(\Sigma)$ to $\Omega^2$ that switches $\Sigma$ to all positive. Consequently, $\sbdim(\Sigma)>2$. Now, the  switching function $\zeta:V(\Sigma)\rightarrow\Omega^3$ defined by $\zeta(v_1)=(1,0,0)$, $\zeta(v_2)=(0,0,1)$, $\zeta(v_3)=(-1,-1,-1)$, $\zeta(v_4)=(0,1,0)$, $\zeta(v_5)=(-1,1,1)$, $\zeta(v_6)=(1,-1,1)$ and $\zeta(v_7)=(1,1,-1)$ is injective, and switches $\Sigma$ to all positive. Hence, $\sbdim(\Sigma)=3$.

 Let $\eta$ be the $1$-switching function defined on $V(\Sigma)$ as follows: $\eta(v_1)=\eta(v_2)=-1$ and $\eta(v_3)=\eta(v_4)=\eta(v_5)=\eta(v_6)=\eta(v_7)=1$. The corresponding switched signed graph $\Sigma^\eta$ is shown in  Figure~\ref{fig3}.

 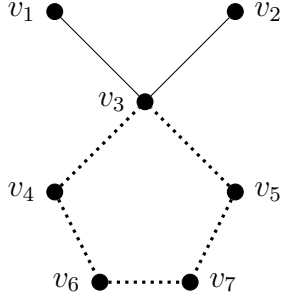
\begin{figure}[h]
 	\centering
 	\[\begin{tikzpicture}[x=0.6cm, y=0.6cm]
 		\vertex[fill] (a) at (-1,6) [label=left:$v_1$] {};
 		\vertex[fill] (b) at (3,6) [label=right:$v_2$] {};
 		\vertex[fill] (c) at (1,4) [label=left:$v_3$] {};
 		\vertex[fill] (d) at (-1,2) [label=left:$v_4$] {};
 		\vertex[fill] (e) at (3,2) [label=right:$v_5$] {};
 		\vertex[fill] (f) at (0,0) [label=left:$v_6$] {};
 		\vertex[fill] (g) at (2,0) [label=right:$v_7$] {};
 		\path[very thick, dotted]

 		(d) edge (c)
 		(e) edge (c)
 		(d) edge (f)
 		(e) edge (g)
 		(f) edge (g)

 		;
 		\path
 	(a) edge (c)
 	(b) edge (c)
 	;
 	\end{tikzpicture}\]
 	\caption{The $1$-switched signed graph $\Sigma^\eta$ with $\sbdim(\Sigma^\eta)=2$}
 	\label{fig3}
 \end{figure}
  A simple computation shows that there is an injective switching function $\zeta':V(\Sigma^\eta)\rightarrow\Omega^2$ defined by $\zeta'(v_1)=(1,1)$, $\zeta'(v_2)=(1,-1)$, $\zeta'(v_3)=(1,0)$, $\zeta'(v_4)=(-1,1)$, $\zeta'(v_5)=(-1,-1)$, $\zeta'(v_6)=(0,-1)$ and $\zeta'(v_7)=(0,1)$.
 Consequently, $\sbdim(\Sigma^\eta)=2$ and hence $\sbdim$ is not $1$-switching invariant.
\end{rmk}

\begin{thm}\label{thm6}
	For any subgraph $\Sigma'$ of $\Sigma$, $\bdim(\Sigma') \leq\bdim(\Sigma)$ and $\sbdim(\Sigma') \leq\sbdim(\Sigma)$.
\end{thm}

 Our next theorem shows how the balancing dimension of a disconnected signed graph depends on its connected components.

\begin{thm}
	The balancing dimension of a disconnected graph is the largest balancing dimension of its connected components.
\end{thm}
\begin{proof}
	Let $\Sigma=(G, \sigma)$ be a signed graph having $t$ components $\Sigma_1,\Sigma_2, \dots, \Sigma_t$.
Let $n$ be the largest balancing dimension of any component $\Sigma_i$.
By Theorem~\ref{thm6}, $\bdim(\Sigma)\geq \bdim(\Sigma_i)$ for all $i$.  Thus, $\bdim(\Sigma)\geq n$.  Since $n\geq\bdim(\Sigma_i)$ for every $i$, by Lemma \ref{up} there exists an $n$-positive function $\zeta_i$  for every component $\Sigma_i$.  Define $\zeta:V(\Sigma)\rightarrow\Omega^n$ by $\zeta(v)=\zeta_i(v)$ if the component that contains vertex $v$ is $\Sigma_i$.
Since each $\zeta_i$ switches $\Sigma_i$ to all positive, $\zeta$ switches $\Sigma$ to all positive. Thus, $\bdim(\Sigma)=n$.
\end{proof}
\begin{rmk}
	The above result will not hold for strong balancing dimension. As an illustration, let us consider $\Sigma$ as the signed graph that consists of $3$ disjoint copies of the negative cycle $C_3^-$. Since $\Sigma$ has $9$ vertices, there does not exist an injective switching function $\zeta:V(\Sigma)\rightarrow\Omega^3$.
\end{rmk}
\begin{thm}\label{thm9}
		Adding pendant edges to a signed graph will not change its balancing dimension.
\end{thm}

%%%%%%%%
\subsection{Bounds for balancing dimensions}\label{sec1}

We begin with upper bounds.  Let $\Sigma=(G,\sigma)$ be a signed graph with $n$ vertices $v_1,v_2,\dots, v_n$ and $m$ edges $e_1, e_2, \dots, e_m$.
For each edge $e_k=v_iv_j$, we define a vector $\textbf{b}(e_k)=\begin{pmatrix}
b_{1k}\\
\vdots	\\
b_{nk}
\end{pmatrix} \in \mathbb{R}^{n\times1}$, whose $i^{\text{th}}$ and $j^{\text{th}}$ entries are $b_{ik}=\pm 1$ and $b_{jk}=b_{ik}\sigma(e_k)$, respectively, and whose other entries are $0$. We now define $B$ as the $n\times m$ matrix whose $k^{th}$ column is the column vector $\textbf{b}(e_k)$; that is, $$B=\begin{bmatrix}
	\textbf{b}(e_1)&\textbf{b}(e_2)  &\cdots  & \textbf{b}(e_m)
\end{bmatrix} = \begin{pmatrix}
b_{ij}
\end{pmatrix}_{n\times m}.$$
The matrix $B$ is precisely an incidence matrix of the signed graph $-\Sigma$.
 We now define a switching function $\mu: V(\Sigma)\rightarrow \Omega^m$ by $$\mu(v_i)=(b_{i1}, b_{i2},\dots, b_{im})$$
 for $i=1,2,\dots, n$. Then, the function $\mu$ satisfies the following properties:
 \begin{description}
 	\item [Property (i):] $\langle \mu(v_i), \mu (v_j)\rangle = \sigma(v_iv_j)$ for all edges $v_iv_j$ in $\Sigma$.
 	\item[Property (ii):] $\|\mu(v_i)\|^2=d(v_i)$ for all vertices $v_i$ in $\Sigma$.
 \end{description}
By Property (i),
\begin{align*}
	\sigma^\mu(v_iv_j) &=\sigma(v_iv_j)\sgn(\langle \mu(v_i), \mu (v_j)\rangle)\\
	&=\sigma(v_iv_j)^2
	=1.
\end{align*}

Thus, for every signed graph $\Sigma$ with $m$ edges, there always exists a switching function $\mu: V(\Sigma)\rightarrow \Omega^m$ that switches $\Sigma$ to all positive, but $\mu$ is not always injective.  This is the first step towards the following theorem.

\begin{thm}\label{thm7}
	Let $\Sigma$ be a signed graph with $m$ edges.  It satisfies the inequalities $1\leq\bdim(\Sigma)\leq\sbdim(\Sigma)$ and $\bdim(\Sigma)\leq m.$  Furthermore, $\sbdim(\Sigma)\leq m$ if $\Sigma$ has at most one isolated vertex and no component that is a positive edge.
	\end{thm}
\begin{proof}
The inequalities $1\leq\bdim(\Sigma)\leq\sbdim(\Sigma)$ follow from the definitions.  We have already shown that $\bdim(\Sigma)\leq m.$

In order to bound $\sbdim(\Sigma)$ we need to know when $\mu$ is not injective.  There are two ways $\mu$ can fail to be injective.  First, since $\mu(v)=\mathbf0$ if $v$ is isolated, $\mu(v)=\mu(w)$ if both $v$ and $w$ are isolated.  Second, if $G$ has a component $K_2$, then $\mu(u)=\sigma(uv)\mu(v)=$ the vector with $1$ in the position of edge $uv$ and $0$ in all other positions, so $\mu(u)=\mu(v)$ if $uv$ is a positive edge.  In all other cases, every vertex has a different set of incident edges so all vectors $\mu(v)$ are distinct.  This proves the third inequality.
 \end{proof}
\begin{ex}
Let $\Sigma$ be the unbalanced signed graph shown in Figure~\ref{fig4}.
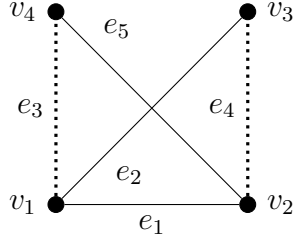
\begin{figure}
	\centering
	\[\begin{tikzpicture}[x=0.85cm, y=0.855cm]
		\vertex[fill] (u) at (0,0) [label=left:$v_1$] {};
		\vertex[fill] (v) at (3,0) [label=right:$v_2$] {};
		\vertex[fill] (w) at (0,3) [label=left:$v_4$] {};
		\vertex[fill] (x) at (3,3) [label=right:$v_3$] {};
		\path[very thick, dotted]
		(u) edge["$e_3$"] (w)
		(v) edge[ "$e_4$"] (x)

		;
		\path
		(x) edge[pos=.76, "$e_2$"] (u)
		(w) edge[pos=.17, "$e_5$"] (v)
		(v) edge ["$e_1$"] (u)
		;
	\end{tikzpicture}\]
	\caption{An illustration}
	\label{fig4}
\end{figure}
Then,
 \begin{align*} B&=\begin{bmatrix}
 		\textbf{b}(e_1)&\textbf{b}(e_2)&\textbf{b}(e_3)&\textbf{b}(e_4)&\textbf{b}(e_5)
 	\end{bmatrix}\\
	&=\begin{bmatrix}
		1&  1&  1& 0 & 0 \\
		\sigma(e_1)&  0& 0 & 1 & 1 \\
		0& \sigma(e_2) & 0 & \sigma(e_4) & 0 \\
		0& 0 & \sigma(e_3) & 0 & \sigma(e_5)
	\end{bmatrix} = \begin{bmatrix}
		1&  1&  1& 0 & 0 \\
		1&  0& 0 & 1 & 1 \\
		0& 1  & 0 & -1 & 0 \\
		0& 0 & -1 & 0 & 1
	\end{bmatrix}
\end{align*}

Define $\mu:V(\Sigma)\rightarrow \Omega^5$ as follows:
\begin{align*}
\mu(v_1)&=(1,1,1,0,0),\\
\mu(v_2)&=(1,0,0,1,1),\\
\mu(v_3)&=(0,1,0,-1,0),\\
\mu(v_4)&=(0,0,-1,0,1).
\end{align*}
 Then $\mu$ is injective and satisfies $\langle \mu(v_i), \mu (v_j)\rangle = \sigma(v_iv_j)$ for all edges $v_iv_j$ in $\Sigma$. Thus $\mu$ switches $\Sigma$ to all positive. Also, we can observe that  $\|\mu(v_i)\|^2=d(v_i)$ for all vertices $v_i$ in $\Sigma$.
\end{ex}

A lower bound exists in terms of the structure of the underlying graph $G$ of $\Sigma$.  The clique number $\omega(G)$ denotes the largest order of a clique in $G$.  Let $\lambda(k)$ denote the largest number of pairwise non-orthogonal lines generated by the vectors in $\Omega^k$.  For instance, $\lambda(2)=2$.  The largest number of pairwise non-orthogonal vectors in $\Omega^k$ equals $2\lambda(k)$.  Computation of $\lambda$ appears to be hard, but $\lambda$ does give a lower bound on balancing dimension.

\begin{thm}\label{cliquebound}
We have $\bdim(\Sigma) \geq \lambda^{-1}(\lceil\frac12\omega(G)\rceil)$.
\end{thm}
\begin{proof}
Let $\zeta$ be a positive $k$-switching function for $\Sigma$.  In a clique of order $p$ all the vectors $\zeta(v)$ for the vertices of the clique must be non-orthogonal.  Therefore, $p \leq 2\lambda(k)$, equivalently $k \geq \lambda^{-1}(\lceil p/2\rceil)$.   Considering a clique of maximum order gives the theorem.
\end{proof}

Negative triangles are important.

\begin{thm}\label{negtriangle}
If $\Sigma$ contains a negative triangle, then $\bdim(\Sigma)\geq3$.
\end{thm}

\begin{proof}
Let $C_3^-$ be a negative triangle in $\Sigma$.  First we prove that $\bdim(C_3^-) \geq 3$.  Suppose $\zeta$ is a $2$-switching function that makes $C_3^-$ all positive.  All vectors $\zeta(v)$ are non-orthogonal because all vertices are adjacent.  There are exactly $4$ lines generated by $\Omega^2$ and only two of them can be chosen to be non-orthogonal. Without loss of generality, let the lines be $x_2=0$ and $x_1=x_2$.  The corresponding vectors are $(1,0)$, $(1,1)$, $(-1,0)$, and $(-1,-1)$.  The first pair has positive inner product and the second pair has positive inner product, but any one of the first pair has negative inner product with each of the second pair.  Therefore, the signs generated by $\zeta$ are the same as the signs generated by the $1$-switching function $\mu$ given by
$$
\mu(v) = \begin{cases}
+1 &\text{if } \zeta(v) \in \{(1,0),\ (1,1)\}, \\
-1 &\text{if } \zeta(v) \in \{(-1,0),\ (-1,-1)\}.
\end{cases}
$$
Thus, $\mu$ is a $1$-switching function that makes $C_3^-$ all positive, hence $\bdim(C_3^-) = 1$, i.e., $C_3^-$ is balanced, contrary to assumption.

Then $\bdim(\Sigma) \geq \bdim(C_3^-) \geq 3$ by Theorem \ref{thm6}.
\end{proof}

There is a simple lower bound on strong balancing dimension.

\begin{thm}\label{logbound}
For a signed graph with $n$ vertices, none of them being isolated, $\sbdim(\Sigma) \geq \log_3(n+1)$.
\end{thm}
\begin{proof}
All vectors $\zeta(v)$ must be distinct and non-zero.  In $\Omega^k$ there are $3^k-1$ distinct non-zero vectors.  Therefore, $n\leq 3^k-1$, from which the result follows.
\end{proof}

It would be interesting to know whether there are many signed graphs for which the lower bound given in the previous theorem is attained, i.e., $\sbdim(\Sigma) = \lceil\log_3(n+1)\rceil$.
We present one such graph.  Consider the unbalanced cycle $C_8^-=v_1e_1v_2\cdots v_8e_8v_1$, having only one positive edge $e_8$. The function $\zeta:V(C_8^-)\rightarrow\Omega^2$ given by $\zeta(v_1)=(1,0)$, $\zeta(v_2)=(-1,0)$, $\zeta(v_3)=(1,-1)$, $\zeta(v_4)=(-1,1)$, $\zeta(v_5)=(0,-1)$, $\zeta(v_6)=(0,1)$, $\zeta(v_7)=(-1,-1)$ and $\zeta(v_8)=(1,1)$ is injective and switches $C_8^-$ to all positive. Thus $\sbdim(C_8^-)\leq2=\log_3(9)$ and equality follows by Theorem \ref{logbound}.

%%%%%%%

%%%%%%%%%%%%%%%%
\section{Some Classes of Signed Graphs}
In this section, we compute the balancing and strong balancing dimensions of certain classes of unbalanced signed graphs.

%%%%%
\subsection{Cycles and wheels}\label{wheel}
\begin{prop}\label{lem1}
For an unbalanced cycle $C_n^-$,
$\bdim(C_n^-)= 3$ if $n=3$, and $2$ if $n>3$.
\end{prop}
\begin{proof}
	 Since balancing dimension is $1$-switching invariant, we have only to consider a signed cycle $C_n^-=v_1e_1v_2\cdots v_ne_nv_1$ where $\sigma(e_n)=-1$ and other edges $e_i$ are all positive.

	 If $n>3$, define $\zeta:V(C_n^-)\rightarrow \Omega^2$ by $\zeta(v_1)=(1,0)$, $\zeta(v_2)=(1,1)$, $\zeta(v_n)=(-1,1)$ and for $i=3,4,\dots, n-1$, $\zeta(v_i)=(0,1)$. This is the required $2$-switching for making $\bdim(C_n^-)=2$.

	 We know $\bdim(C_3^-) \geq 3$ by Theorem \ref{negtriangle}.
	Define  $\zeta:V(C_3^-)\rightarrow \Omega^3$ by $\zeta(v_1)=(1,0,0)$, $\zeta(v_2)=(1,1,1)$ and $\zeta(v_3)=(-1,1,1)$. This $3$-switching function shows that $\bdim(C_3^-)\leq3$.
\end{proof}

\begin{rmk}
	The switching function $\zeta$ defined for $C_3^-$ in Proposition~\ref{lem1} is injective and hence  $\sbdim(C_3^-)=3$. If $C_3^-$ is all negative, then the injective switching function  $\zeta:V(C_3^-)\rightarrow \Omega^3$ by $\zeta(v_1)=(-1,1,1)$, $\zeta(v_2)=(1,-1,1)$ and $\zeta(v_3)=(1,1,-1)$ switches $C_3^-$ to all positive.
	Thus, $C_3^-$ gives us an example for which the bound given in Theorem~\ref{thm7} is attained.
\end{rmk}

\begin{ex}\label{ex2}
The balancing dimension  of a signed unicyclic graph $\Sigma$ is that of the unique cycle $C$ in it. Let $\zeta:V(C)\rightarrow \Omega^k$ be the switching function for $C$. We can extend $\zeta$ to $V(\Sigma)$ by adding pendant edges. Hence, by Theorem~\ref{thm9},  $\bdim(\Sigma)=\bdim(C)$.
\end{ex}

The wheel with $n$ spokes is denoted by $W_{n+1}$.
\begin{prop}\label{prop1}
	For an antibalanced signed wheel $W_{n+1}^-$ with $n\geq3$, $\bdim(W_{n+1}^-)=3$.
\end{prop}
\begin{proof}
	Since balancing dimension is $1$-switching invariant, we let $(W_{n+1}^-,\sigma)=(C_n\vee K_1,\sigma)$ with the sign function $\sigma$ given by  $\sigma(e)=-1$ if and only if $e\in E(C_n)$.
	 Let $C_n=v_1v_2\cdots v_n$ and $v_{n+1}=K_1$.
	 Define $\zeta:V(W_{n+1}^-)\rightarrow \Omega^3$ as follows: Choose $\zeta(v_{n+1})=(1,1,1)$. If $n=3k$ or $3k+2$,  assign $\zeta(v_{1})=(-1,1,1)$ and for $i=2,3,\dots, n$, $\zeta(v_i)$ is obtained by performing one left circular shift to $\zeta(v_{i-1})$. If  $n=3k+1$,
	 assign $\zeta(v_{1})=(-1,1,1)$, $\zeta(v_{n})=(1,1,-1)$ and for $i=2,3,\dots, n-1$, $\zeta(v_i)$ is obtained by performing one left circular shift to $\zeta(v_{i-1})$.
\end{proof}

\begin{rmk}
	For the antibalanced signed wheel $W^-_4$ defined above,  $\sbdim(W_4^-)=3$ since the switching function $\zeta$ defined in the proof of Proposition~\ref{prop1} is injective.
\end{rmk}

%%%%%
\subsection{Complete graphs and antibalanced signed graphs}\label{complete}
 We now focus on the balancing dimension of unbalanced signed complete graphs. Since any unbalanced signed complete graph $\Sigma$ contains $C_3^-$ as a subgraph, $\bdim(\Sigma)\geq3$. The following is an example in which the lower bound for balancing dimension is attained.
\begin{ex}\label{knone}
	Let $\Sigma$ be a signed complete graph with $n\ge3$ vertices $v_1, v_2, \dots, v_n$ and having only one negative edge, say $v_1v_n$. Then $\Sigma$ is unbalanced.	Define  $\zeta:V(\Sigma)\rightarrow \Omega^3$ by, $\zeta(v_{1})=(-1,1,1)$, $\zeta(v_{n})=(1,1,-1)$ and for $i=2,3,\dots, n-1$, $\zeta(v_i)=(1,1,1)$. Then, $\zeta$ switches $\Sigma$ to all positive and hence $\bdim(\Sigma)=3$.
\end{ex}

\begin{ex}\label{lem3}
We provide a class of signed graphs in which the balancing dimension and strong balancing dimension coincide:
if $\Sigma$ is an all  negative signed complete  graph, then  $\bdim(\Sigma)=\sbdim(\Sigma)$.

	For the proof suppose $\bdim(\Sigma)=n$ and let $\zeta:V(\Sigma)\rightarrow \Omega^n$ be the corresponding switching function. If $\zeta(v_i)=\zeta(v_j)$ for some $i\neq j$, then $\sgn (\langle \zeta(v_i), \zeta (v_j)\rangle) =+1$ and hence $\sigma^\zeta(v_iv_j)=\sigma(v_iv_j) \sgn (\langle \zeta(v_i), \zeta (v_j)\rangle)=-1$, which is a contradiction. Thus $\zeta$ is injective and hence $\bdim(\Sigma)=\sbdim(\Sigma)$.

Note that these are not the only signed graphs satisfying $\bdim(\Sigma)=\sbdim(\Sigma)$ (see Example~\ref{ex1}).
\end{ex}

\begin{ex} The relationship between balancing dimensions of $\Sigma$ and $-\Sigma$ is an obvious question. We found that there exist signed graphs satisfying $\bdim(-\Sigma)=\bdim(\Sigma)$. Similarly, there exist signed graphs satisfying $\bdim(-\Sigma)\neq\bdim(\Sigma)$.

		(i) Every bipartite signed graph $\Sigma$ satisfies $\bdim(-\Sigma)=\bdim(\Sigma)$ since $\Sigma$ can be $1$-switched to $-\Sigma$ and $\bdim$ is $1$-switching invariant.

		The result doesn't hold for the $\sbdim$. For example, if we consider $\Sigma$ as the all positive tree with $3$ vertices, then $\sbdim(\Sigma)=1$ and $\sbdim(-\Sigma)=2$.

		(ii) Let $\Sigma$ be an odd unbalanced cycle. Then $-\Sigma$ is balanced and hence $\bdim(-\Sigma)=1<\bdim(\Sigma)$.
\end{ex}

\begin{defn}[\cite{nip}]  Let $W$ be a nonempty subset of a vector space over the field of real numbers. $W$ is called a negative inner product (NIP) set if $\langle \alpha, \beta \rangle < 0$ for all $\alpha$ and $\beta$ in $W$ with $\alpha \neq \beta$.
\end{defn}

\begin{lem} [\cite{nip}]\label{lem4}
	In a $ k $-dimensional vector space, there are at most $ k+1 $ vectors in an NIP set.
\end{lem}

 \begin{defn}
 We define $\nu(k)$ to be the largest size of an NIP set in $\Omega^k$.  Thus, $\nu(k)\leq k+1$.  We define $\bar\nu(n) = \min\{k : \nu(k)\geq n\}$.
 \end{defn}

It is easy to see that $\nu(2)=2$ but $\nu(3)$ is not as easy to determine.  Gary Greaves kindly provided the following numbers, which he computed with SageMath:

\begin{tabular}{c|r|r|r|r|r|r|r}
$k$ & 2 & 3 & 4 & 5 & 6 & 7 \\
\hline
$\nu(k)$ & 2 & 4 & 4 & 5 & 5& 8
\end{tabular}

\noindent
From this table we obtain values of  $\bar\nu(n)$:

\begin{tabular}{c|r|r|r|r|r|r|r|}
$n$ & 2 & 3 & 4 & 5 & 6 & 7 & 8\\
\hline
$\bar\nu(n)$ & 2 & 3 & 3 & 5 & 5 & 5 & 7
\end{tabular}
\\

\begin{lem}\label{NIPnu}
$\nu(k)\geq n$ if and only if $k \geq \bar\nu(n)$.  In particular, $\bar\nu(n) \geq n-1$.
\end{lem}
\begin{proof}
We restate the definition of $\bar\nu(n)$ as the minimum $k$ such that there exists an NIP set of $n$ elements in $\Omega^k$.  Thus, $k\geq\bar\nu(n)$ if and only if an NIP set of size $n$ exists in $\Omega^k$.  This is equivalent to saying that $n \leq \nu(k)$.

Choosing $k=n-1$, we have $\nu(n-1)\leq n$ so, equivalently, $\bar\nu(n) \geq n-1$.
\end{proof}

\begin{thm}\label{thm8}
	Let $\Sigma$ be an antibalanced signed complete graph on $n$ vertices, where $n\geq2$. Then
	$\bdim(\Sigma) = \bar\nu(n) \geq n-1$.
\end{thm}
\begin{proof}
	By $1$-switching as necessary assume $\Sigma$ is all negative. For switching $\Sigma$ to all positive, we must assign each vertex of $\Sigma$ one element from an NIP set with elements in $\Omega^k$ for some $k$. Thus, by Lemma \ref{lem4}, it is necessary and sufficient that $n\leq \nu(k)$; equivalently by Lemma \ref{NIPnu}, $\bar\nu(n) \leq k$.  It follows that the smallest possible $k$ is $\bar\nu(n) \geq n-1$.
\end{proof}

\begin{ex}\label{knsw}
	Let $\Sigma$ be the all-negative signed complete graph on $5$ vertices. 
	Then by Theorem~\ref{thm8}, we have $\bdim(\Sigma)\geq 4$. Let us define a $3$-switching function $\mu:V(\Sigma)\rightarrow \Omega^3$ as follows. Let $\mu(v_1)=\mu(v_2)=\mu(v_3)=(1,1,1)$,  $\mu(v_4)=(-1,1,-1)$ and $\mu(v_5)=(-1,-1,1)$. We will show that the switched signed graph $\Sigma^\mu$ has balancing dimension $3$. Since $\Sigma^\mu$ contains $C_3^-$ as a subgraph, $\bdim(\Sigma^\mu)\geq 3$. Now, the function $\zeta:V(\Sigma^\mu)\rightarrow\Omega^3$, defined by $\zeta(v_1)=(1,1,-1)$, $\zeta(v_2)= (-1,1,1)$, $\zeta(v_3)=(1,-1,1)$ and  $\zeta(v_4)= \zeta(v_5)=(1,1,1)$ switches $\Sigma^\mu$ to all positive. Hence $\bdim(\Sigma^\mu)=3$.
\end{ex}
 This example leads us to the following conclusions.
 \begin{enumerate}
 	\item Though balancing dimension is $1$-switching invariant, the same need not be true for a general $k$-switching, where $k\geq2$.
 	\item If $\bdim(\Sigma)=n$, then for $k=2,3,\dots, n-1$, $k$-switching need not leave the  signs of all cycles unchanged.
 \end{enumerate}

%%%%%

\section*{Acknowledgements}
The second author would like to acknowledge his gratitude to the University Grants Commission (UGC), India, for providing financial support in the form of Junior Research fellowship (NTA Ref.\ No.: 191620039346).  The last author extends his gratitude to Christopher Eppolito for valuable discussion.

The authors express their gratitude to Gary Greaves for computing values of $\nu(k)$ in section \ref {complete}.

%%%%%%%%%%%%%%%%

\end{document}